\documentclass[11pt,reqno,a4paper]{amsart}

\usepackage{etoolbox}
\usepackage{amsmath}
\usepackage{enumerate}
\usepackage{amssymb}
\usepackage[initials]{amsrefs}
\usepackage{graphicx,latexsym,amsbsy,amscd,amsthm,epsfig,verbatim}
\usepackage{amsfonts}
\usepackage{comment}
\usepackage{courier} 
\usepackage[T1]{fontenc}                                                    
\usepackage{hyperref}
\usepackage{slashed}
\usepackage{tikz-cd}

\usepackage{todonotes}
\usepackage{mathtools}
\usepackage{makecell}
\usepackage{bm}
\usepackage{wasysym}

\newcommand{\R}{\mathbb{R}}
\newcommand{\C}{\mathbb{C}}

\newcommand\subsetsim{\mathrel{%
\ooalign{\raise0.2ex\hbox{$\subset$}\cr\hidewidth\raise-0.8ex\hbox{\scalebox{0.9}{$\sim$}}\hidewidth\cr}}}

\DeclareMathOperator{\rank}{rank}

\newcommand{\Z}{\mathbb Z}

\newtheorem{theorem}{Theorem}

\newtheorem{proposition}[theorem]{Proposition}
\newtheorem{lemma}[theorem]{Lemma}

\theoremstyle{definition}
\newtheorem{definition}[theorem]{Definition}

\newtheorem{remark}[theorem]{Remark}
\newtheorem{example}[theorem]{Example}

\patchcmd{\subsection}{-.5em}{.5em}{}{}
\patchcmd{\subsubsection}{-.5em}{.5em}{}{}

\title{Products of free groups in Lie groups}

\author{Caterina Campagnolo}
\email{caterina.campagnolo@ens-lyon.fr}

\address{Unité de mathématiques pures et appliquées, ENS Lyon, France}

\author{Holger Kammeyer}
\email{holger.kammeyer@kit.edu}

\address{Institute for Algebra and Geometry, Karlsruhe Institute of Technology, Germany}

\subjclass[2010]{22E15, 17B20}

\begin{document}
\begin{abstract}
For every Lie group \(G\), we compute the maximal~\(n\) such that an \(n\)-fold product of nonabelian free groups embeds into \(G\).
\end{abstract}

\maketitle

\section{Introduction}

Let \(G\) be a group.  If \(G\) is abelian, the \emph{rank} of \(G\), if finite, can be characterized as the maximal \(n\) such that \(\mathbb{Z}^n\) embeds as a subgroup of \(G\).  This suggests the following definition as a noncommutative companion.

\begin{definition}
We define the \emph{free subgroup rank} of \(G\) by
\[
\nu (G)= \max\{n \ge 0 \colon (F_2)^n \text{ embeds into } G \} \in \mathbb{N}_0 \cup \{ \infty \}.
\]
\end{definition}

Since \(F_2\) contains free groups on arbitrary many letters as subgroups, we observe that any other choice of noncommutative free groups as factors of an \(n\)-fold product leads to the same notion of free subgroup rank. 

The purpose of this paper is to compute \(\nu(G)\) for any almost connected \emph{Lie group} \(G\).  As we explain below, the task is easily reduced to the case when \(G\) is connected and simple.  The Lie algebra of \(G\) then either is \emph{absolutely simple}, meaning it has simple complexification, or if not, it carries itself the structure of a complex simple Lie algebra.  In either case, we obtain an associated irreducible root system \(\Phi\).  Recall that two roots \(\alpha, \beta \in \Phi\) are called \emph{strongly orthogonal} if they are orthogonal and \(\alpha \pm \beta \notin \Phi\).

\begin{definition}
The \emph{strong orthogonal rank} \(\operatorname{sork}(\Phi)\) of a root system \(\Phi\) is the maximal cardinality of a set of pairwise strongly orthogonal roots in \(\Phi\).
\end{definition}

As our main result, it turns out that in almost all cases \(\nu(G) = \operatorname{sork}(\Phi)\).

\begin{theorem} \label{thm:nu-of-simple}
Let \(G\) be a connected simple Lie group and let \(\mathfrak{g}_0\) be the Lie algebra of \(G\) with complexification \(\mathfrak{g} = \mathfrak{g}_0 \otimes \C\).
\begin{enumerate}[(i)]
\item \label{item:complex-structure} If \(\mathfrak{g}_0\) has itself the structure of a complex Lie algebra, then
\[ \nu(G) = \operatorname{sork} \Phi(\mathfrak{g}_0). \]
\item \label{item:real-form} If \(\mathfrak{g}_0\) has no complex structure, then
\[ \nu(G) = \operatorname{sork} \Phi(\mathfrak{g}),\, \]
unless case~\eqref{item:sopq} applies.
\item \label{item:sopq} If \(\mathfrak{g}_0 \cong \mathfrak{so}(p,q)\) for odd \(p, q\) such that \(p + q\) is divisible by four, then
\[ \quad \ \nu(G) = \operatorname{sork} \Phi(\mathfrak{g}) - 1. \]
\end{enumerate}
\end{theorem}

One may argue that the exceptional case~\eqref{item:sopq} had to be expected in view of the accidental isomorphism \(\mathfrak{so}(3,1) \cong \mathfrak{sl}_2(\C)\).  For we have \(\nu(\operatorname{SL}_2(\C)) = 1\) by~\eqref{item:complex-structure} whereas \(\operatorname{sork}(D_2) = \operatorname{sork}(A_1 \times A_1) = 2\).  To compute \(\nu(G)\) explicitly, it remains to list the strong orthogonal ranks of irreducible root systems.

\begin{proposition}\label{prop:sork-of-irred}
The strong orthogonal rank of an irreducible root system~\(\Phi\) is given by the following table.
\[
\renewcommand{\arraystretch}{1.2}
\begin{array}{lccccccccccc}
\hline
\textup{\textbf{Type of }} \bm{\Phi} & A_{2r-1} & A_{2r} & B_r & C_r & D_{2r+1} & D_{2r} & E_6 & E_7 & E_8 & F_4 & G_2\\
\textup{\textbf{sork}} \,\bm{\Phi} & r & r & r & r & 2r & 2r & 4 & 7 & 8 & 4 & 2  \\
\hline
\end{array}
\]
\end{proposition}

One finds these values routinely by exhibiting maximal sequences of pairwise strongly orthogonal roots in \(\Phi\), for example using the explicit descriptions given in~\cite{Humphreys:Lie-algebras}*{Section~12.1, p.\,63}.  Note that the strong orthogonal rank equals the rank of the root system whenever the Dynkin diagram has no symmetries.  For Dynkin diagrams with symmetries, the strong orthogonal rank is strictly smaller than the rank with the exception of \(D_{2r}\) which is also the problematic type occurring in case~\eqref{item:sopq}.

\begin{example}
Up to coverings, the complete list of connected simple Lie groups~\(G\) with \(\nu(G) = 1\) is given by
\[ \mathrm{SU}(2),\, \mathrm{SL}_2(\R),\, \mathrm{SL}_2(\C),\, \mathrm{SU}(3),\, \mathrm{SU}(2,1),\, \mathrm{SL}_3(\R),\, \mathrm{SL}_3(\C). \]
\end{example}

\begin{example}
An example sequence of Lie groups in which also the exceptional case becomes relevant is given by the isometry groups of hyperbolic space \(G = \mathrm{SO}(n,1)\).  Here Theorem~\ref{thm:nu-of-simple} gives explicitly \(\nu(\mathrm{SO}(n,1)) = \lfloor \frac{n}{2} \rfloor\).
\end{example}

Now let \(G\) be a general Lie group.  To avoid trivialities like considering an arbitrary infinite group as a zero dimensional Lie group, we only assume that \(G\) is \emph{almost connected}, meaning it has finitely many path components.  We then have \(\nu(G) = \nu(G^0)\) where \(G^0\) is the unit component of \(G\) because \(\nu(G)\) remains unchanged when passing to finite index subgroups (Lemma~\ref{lemma:commensurable}).  So we may assume \(G\) is connected to begin with.  The kernel of the universal covering projection \(p \colon \widetilde{G} \rightarrow G\) is a normal and discrete subgroup in a path connected group, hence it is central.  Thus \(\widetilde{G}\) is a central extension of \(G\) and we have \(\nu(G) = \nu(\widetilde{G})\) by Proposition~\ref{prop:solvable-extensions}.  So we may assume \(G\) is simply-connected.  Let \(R \subset G\) be the \emph{solvable radical} of \(G\), meaning the maximal closed connected normal solvable subgroup.  The Levi-Malcev decomposition asserts that there exists a closed simply-connected semisimple subgroup \(S \subset G\) such that \(G=RS\) is a semidirect product~\cite{Onishchik-Vinberg:Lie-groups}*{Theorem~2, Chapter~6, p.\,284}.  Since solvable groups cannot contain \(F_2\), we have \(\nu(R) = 0\) so that Proposition~\ref{prop:solvable-extensions} implies \(\nu(G) = \nu(S)\).  Thus we may assume \(G\) is simply-connected semisimple.  In that case \(G\) decomposes as a direct product \(G = G_1 G_2 \cdots G_k\) of connected simple Lie groups.  By Proposition~\ref{prop:products}\,\eqref{item:direct-product}, we have \(\nu(G) = \nu(G_1) + \cdots + \nu(G_k)\).  Theorem~\ref{thm:nu-of-simple} completes the computation of \(\nu(G)\).

\medskip
For the convenience of the reader, we give an overview of the proof of Theorem~\ref{thm:nu-of-simple}.  We start with part~\eqref{item:complex-structure}.  To see the inequality \(\nu(G) \ge \operatorname{sork} \Phi(\mathfrak{g}_0)\), one notices that \(n\) strongly orthogonal roots in \(\Phi(\mathfrak{g}_0)\) span a closed subroot system \(\Sigma \subset \Phi(\mathfrak{g}_0)\) of type \((A_1)^n\).  Correspondingly, we obtain a subalgebra \((\mathfrak{sl}_2(\C))^n \subset \mathfrak{g}_0\) and hence a subgroup in \(G\) isogenous to \((\operatorname{SL}_2(\C))^n\).  This group in turn contains \((F_2)^n\).  The harder part is the reverse inequality \(\nu(G) \le \operatorname{sork} \Phi(\mathfrak{g}_0)\).  After factoring out the center, the adjoint representation realizes \(G\) as a complex linear algebraic group.  Every \(F_2\)-factor of a subgroup \((F_2)^n \subset G\) has non-solvable Zariski closure in this group, hence the corresponding Lie subalgebra of \(\mathfrak{g}_0\) contains a semisimple Levi subalgebra and in turn a copy of \(\mathfrak{sl}_2(\C)\).  In this manner, we obtain a subalgebra \((\mathfrak{sl}_2(\C))^n \subset \mathfrak{g}_0\).  However, not every semisimple subalgebra of a semisimple Lie algebra is \emph{regular} in the sense that it comes from a closed subroot system \(\Sigma \subset \Phi(\mathfrak{g}_0)\): just think of \(\mathfrak{so}(n) \subset \mathfrak{sl}_n\).  So it remains to see that regular maximal subalgebras of the form \((\mathfrak{sl}_2(\C))^n \subset \mathfrak{g}_0\) are overall maximal.  This is the statement of Proposition~\ref{prop:regular-versus-not}, and its verification forms a technical core part of the proof, as it requires a delicate study of Dynkin's classical work on the classification of semisimple subalgebras of semisimple Lie algebras~\citelist{\cite{Dynkin:subalgebras} \cite{Dynkin:maximal-subgroups}}.  Once the proposition is proven, we know that \(\Phi(\mathfrak{g}_0)\) contains a closed subroot system \(\Sigma\) of type \((A_1)^n\) and picking one root from each \(A_1\)-factor yields \(n\) pairwise strongly orthogonal roots in \(\Phi(\mathfrak{g}_0)\).

\smallskip
The proof, as just described, carries over to part~\eqref{item:real-form} of the theorem whenever \(\mathfrak{g}_0\) is a \emph{split real form} of \(\mathfrak{g}\), meaning for any Cartan decomposition \(\mathfrak{g}_0 = \mathfrak{k}_0 \oplus \mathfrak{p}_0\), there exists a Cartan subalgebra \(\mathfrak{h}_0\) of \(\mathfrak{g}_0\) contained in \(\mathfrak{p}_0\).  In general, we still obtain \(\operatorname{sork} \Phi(\mathfrak{g})\) as an upper bound for \(\nu(G)\) because \(G\) embeds into its complexification \(G_\C\).  For the lower bound, however, we need additional arguments.  We start with the case opposite to the split case, when \(\mathfrak{g}_0\) has a Cartan subalgebra \(\mathfrak{h}_0 \subset \mathfrak{k}_0\).  This means the complexified Cartan involution \(\theta\) acts trivially on the root system \(\Phi(\mathfrak{g})\) which becomes visible in a \emph{Vogan diagram} of \(\mathfrak{g}_0\) without arrows.  Therefore \(\theta\) leaves each root space \(\mathfrak{g}_\alpha\) invariant, and the root \(\alpha\) is termed \emph{compact} or \emph{noncompact} according to whether \(\theta_{|\mathfrak{g}_\alpha} = \mathrm{id}\) or \(\theta_{|\mathfrak{g}_\alpha} = -\mathrm{id}\).  Compact roots give rise to an \(\mathfrak{su}(2)\)-subalgebra and noncompact roots to an \(\mathfrak{sl}_2(\R)\)-subalgebra of \(\mathfrak{g}_0\).  Hence, \(n\) strongly orthogonal roots in \(\Phi(\mathfrak{g})\), of which \(k\) are compact and \((n-k)\) are noncompact, give rise to a regular semisimple subalgebra of \(\mathfrak{g}_0\) consisting of \(k\) different \(\mathfrak{su}(2)\)-ideals and \((n-k)\) different \(\mathfrak{sl}_2(\R)\)-ideals.  Since both \(\operatorname{SL}_2(\R)\) and \(\operatorname{SU}(2)\) contain \(F_2\) as subgroup, we obtain a subgroup \((F_2)^n\) in \(G\).

\smallskip
Up to isomorphism, the real simple Lie algebras \(\mathfrak{g}_0\) which are neither split nor have a compact Cartan subalgebra fall in one of only three different families, two of which can easily be dealt with by hand.  The only remaining case needing special attention is \(\mathfrak{g}_0 \cong \mathfrak{so}(p,q)\) for odd integers \(p\) and \(q\) with \(p+q = 0 \text{ mod } 4\).  This brings us to part~\eqref{item:sopq} of the theorem.  In this case, we may assume \(G = \operatorname{SO}^0(p,q)\)  and the maximal compact subgroup \(K = \operatorname{SO}(p) \times \operatorname{SO}(q)\) has \(\nu(K) = \frac{p+q}{2} - 1 = \operatorname{sork} \Phi(\mathfrak{g}) - 1\) by part~\eqref{item:real-form}.  Hence it remains to exclude the possibility \(\nu(G) = \operatorname{sork} \Phi(\mathfrak{g})\).  If there was \((F_2)^n \subset G\) with \(n = \operatorname{sork} \Phi(\mathfrak{g})\), we would obtain a closed \(\theta\)-invariant subroot system \(\Sigma \subset \Phi(\mathfrak{g})\) of type \((A_1)^n\) from the Zariski closure of \((F_2)^n\) because semisimple subalgebras of full rank are automatically regular (Proposition~\ref{prop:full-rank-regular}).  Moreover, \(\theta\) cannot swap any two \(A_1\)-factors of \(\Sigma\) because then this \((A_1)^2\)-subsystem would correspond to an \(\mathfrak{sl}_2(\C)\)-subalgebra considered as real algebra which has free subgroup rank one by part~\eqref{item:complex-structure}.  So the subgroup corresponding to \(\Sigma\) could have free subgroup rank at most~\(n-1\), which is a contradiction.  Hence \(\theta\) maps each root of \(\Sigma\) either to itself or to its opposite and the roots are called \emph{imaginary} and \emph{real} accordingly.  By a sequence of \emph{Cayley transforms} corresponding to the real roots, one can find a new Cartan subalgebra \(\mathfrak{h}_0\) of \(\mathfrak{g}_0\) in which all roots have become imaginary.  But that would mean \(\mathfrak{h}_0 \subset \mathfrak{k}_0\), which is absurd because \(\mathfrak{so}(p,q)\) has no compact Cartan subalgebra.

\medskip
A notion related to free subgroup rank can be found in work of Ab\'ert \cite{Abert:non-commuting}*{Corollary~6}: every linear group \(G\) has a finite maximal number of pairwise commuting nonabelian subgroups.  Of course \(\nu(G)\) is a lower bound for this number.  More precisely, the Tits alternative shows that \(\nu(G)\) is the maximal number of pairwise commuting finitely generated non-virtually solvable subgroups \(\Gamma_i\) of \(G\) with pairwise trivial intersection.  The Breuillard--Gelander uniform sharpening of the Tits alternative~\cite{Breuillard-Gelander2008} gives moreover constants \(m_i\) for each \(\Gamma_i\), such that for any symmetric generating set \(\Sigma_i\) of \(\Gamma_i\), two words in at most \(m_i\) letters from \(\Sigma_i\) generate a free subgroup \(F_2 \le \Gamma_i\).  These remarks still apply for Lie groups \(G\) because the relevant semisimple group \(S\), associated with \(G\) by the Levi-Malcev decomposition above, is a central extension of a linear group by means of the adjoint representation.  In the particular cases \(G = \mathrm{GL}_d(\mathbb{C})\) or \(G = \mathrm{SL}_d(\mathbb{C})\), computing \(\nu(G)\) is sort of the reverse problem to determining the minimal dimension of a faithful representation for \((F_2)^n\).  It is an old problem whether a given group \(G\) is linear, or more precisely to find the minimal dimension \(m_K(G)\) of a faithful representation of \(G\) over a field \(K\).  Button~\cite{Button} contributes to this question in the context of right angled Artin groups and free-by-cyclic groups.  For example, \cite{Button}*{Corollary 2.4} gives \(m_K((F_2)^2) = 4\).  By definition, \(m_\mathbb{C}((F_2)^n)\leq d\) if and only if \(\nu(\mathrm{SL}_d(\mathbb{C}))\geq n\). Hence \(\nu(\mathrm{SL}_3(\mathbb{C})) = 1\).  More generally, one can show that  \(m_K\) is additive for products if $K$ is algebraically closed.  If in addition \(\mathrm{char}(K) = 0\), we thus obtain \(m_K((F_2)^n)=2n\) (cf.\,Example \ref{example:nrank}\ref{item:sl2z}).  Therefore the results \(\nu(\mathrm{SL}_{2r}(\mathbb{C}))=\nu(\mathrm{SL}_{2r+1}(\mathbb{C}))= r\) from Theorem \ref{thm:nu-of-simple} and Proposition \ref{prop:sork-of-irred} are as expected.

\medskip
The outline of this article is as follows.  Section~\ref{section:free-rank} gives examples and general properties of the invariant \(\nu(G)\).  Section~\ref{section:lie-algebras} collects some results on semisimple Lie algebras and their semisimple subalgebras which we will use in the subsequent discussion.  Section~\ref{section:maximal} is dedicated to the proof of Proposition~\ref{prop:regular-versus-not} which, as explained, is key for our considerations.  Finally, Section~\ref{section:simple} gives the detailed proof of Theorem~\ref{thm:nu-of-simple} along the above lines.  The authors acknowledge financial support by the Deutsche Forschungsgemeinschaft (DFG, German Research Foundation) 281869850 (RTG 2229).  The first author is moreover grateful to the Swiss National Science Foundation for support in the form of grant P400P2-191107 while the second author is grateful for additional support by DFG 338540207 (SPP 2026).  We are indebted to Yves de Cornulier and Dominik Francoeur for helpful discussions and to Martin Bridson whose talk at the Oberwolfach workshop ``Manifolds and Groups'' in February 2020 gave the impetus to this article.

\section{Free subgroup rank}
\label{section:free-rank}

In this section, we collect some properties of the free subgroup rank which we will exploit in our proofs.  We will refer to a subgroup \((F_2)^n \le G\) as a \emph{free \(n\)-torus} in \(G\).

\begin{example} \label{example:nrank}
To gain familiarity with \(\nu(G)\), we start with some examples.
  \begin{enumerate}[(i)]
  \item We have \(\nu(G) = 0\) if \(G\) is an \emph{amenable} discrete group.  Whether the converse holds, was a famous question by von Neumann, answered in the negative by Ol'shanskii in 1980 using Tarski monsters.  Later, finitely presented and even type \(F_\infty\) counterexamples arose from the work of N.\,Monod, Y.\,Lodha, and J.\,T.\,Moore \cites{Monod:groups of piecewise projective homeomorphisms, Lodha-Moore:nonamenable finitely presented group}.
\item On the other hand, the Tits alternative implies \(\nu(G) \ge 1\) if \(G\) is a finitely generated non-virtually solvable matrix group over some field.
\item \label{item:sl2z} As an example of the last point, we have \(\nu(\mathrm{SL}_2(\mathbb{Z}))=1\).  One can play ping pong to show that \(\begin{psmallmatrix} 1 & 2 \\ 0 & 1 \end{psmallmatrix}\) and \(\begin{psmallmatrix} 1 & 0 \\ 2 & 1 \end{psmallmatrix}\) generate a free subgroup of finite index and use Lemma \ref{lemma:commensurable} below.
\item It is well-known that every non-elementary \emph{Gromov hyperbolic} group \(G\) contains a nonabelian free group.  On the other hand, \(G\) cannot contain~\(\Z^2\), hence \(\nu(G) = 1\).
\item \label{item:slnz} We have \(\nu(\mathrm{SL}_n(\Z)) = \lfloor \frac{n}{2} \rfloor\).  The ``\(\ge\)''-part follows from fitting the virtually free groups \(\mathrm{SL}_2(\Z)\) along the diagonal.  The ``\(\le\)''-part follows for example as a special case of Theorem~\ref{thm:nu-of-simple}\,\eqref{item:real-form}.
    \item We have \(\nu(\operatorname{Aut}(F_n)) = 2n-3\). In fact, M.\,Bridson and R.\,Wade classify maximal free tori in \(\operatorname{Aut}(F_n)\) in yet unpublished work \cite{Bridson}.
    \end{enumerate}
\end{example}

\begin{lemma} \label{lemma:commensurable}
If \(G_1 \le G_2\), then \(\nu(G_1) \le \nu(G_2)\) with equality if \([G_2 : G_1] < \infty\).
\end{lemma}

\begin{proof}
  Only the additional assertion needs a little argument:  For any free torus $(F_2)^n\leq G_2$, let $F_{2,j}\leq (F_2)^n$ denote the various factors for $1\leq j\leq n$.  Then by M.\,Hall's lemma~\cite{Hall:topology-for-free-groups}*{Section 2, Property 1}, $F_{2, j}\cap G_1$ is a finite index subgroup of \(F_{2,j}\), hence a nonabelian free group for all $1\leq j\leq n$. Consequently, $G_1$ contains $(F_2)^n$ as subgroup as well.
\end{proof}

In particular, \(\nu(G)\) is an invariant of the abstract commensurability class of \(G\).  Since the \(n\)-th Betti number of \((F_2)^n\) equals \(2^n\), the lemma also implies that if \(G\) has torsion-free subgroups of finite index, then
\[ \nu(G) \le \operatorname{vcd}(G) \]
for the \emph{virtual cohomological dimension} of \(G\).  Moreover:

\begin{proposition} \label{prop:products}
  For the direct and free products of groups, we have
  \begin{enumerate}[(i)]
  \item \label{item:direct-product} \(\nu(G_1 \times G_2) = \nu(G_1) + \nu(G_2)\),
  \item \label{item:free-product} \(\nu(G_1 \ast G_2) = \max\{1, \nu(G_1),\nu(G_2)\}\) provided  \(G_1\) and \(G_2\) are nontrivial and not both of order two.
  \end{enumerate}
\end{proposition}

\begin{proof}
  The ``\(\ge\)''-part of \eqref{item:direct-product} is trivial.  For the ``$\leq$''-part, suppose there is a subgroup $K\leq G_1 \times G_2$ isomorphic to \(F_2\) that projects to both $G_1$ and $G_2$ with nontrivial kernel.  Picking a nontrivial element from each kernel gives two nontrivial commuting elements in \(K\) which do not lie in a cyclic subgroup of \(K\).  This is a contradiction to \(K\) being nonabelian free.  Hence each \(F_2\)-subgroup in \(G_1 \times G_2\) projects injectively either to \(G_1\) or to \(G_2\) (or to both).  If two commuting subgroups \(K_1\) and \(K_2\) of \(G_1 \times G_2\) project injectively into the same factor \(G_i\), then their images still commute.  If \(K_1\) and \(K_2\) are in addition nonabelian free, then the two images must have trivial intersection because nonabelian free groups are center-free.  Altogether, we conclude that for any subgroup \((F_2)^n \leq G_1 \times G_2\), there exists \(0 \le k \le n\) such that \((F_2)^k\) is a subgroup of \(G_1\) and \((F_2)^{n-k}\) is a subgroup of \(G_2\).  This shows~\eqref{item:direct-product}.

  It is again trivial that \(\nu(G_1 \ast G_2) \ge \max\{\nu(G_1),\nu(G_2)\}\).  By the Kurosh subgroup theorem, any subgroup of \(G_1 \ast G_2\) is itself a free product whose free factors consist of a free group and of conjugates of subgroups of \(G_1\) and \(G_2\).  Therefore any free \(n\)-torus in \(G_1 \ast G_2\) is either one-dimensional or conjugate to a subgroup of \(G_1\) or \(G_2\).  The assumptions assure moreover that \(G_1 \ast G_2\) contains a nonabelian free subgroup.  This shows~\eqref{item:free-product}.
\end{proof}

As we saw in the introduction, it is essentially a consequence of the following general proposition that the free subgroup rank of a Lie group equals the free subgroup rank of any Levi subgroup.

\begin{proposition} \label{prop:solvable-extensions}
Consider an extension 
\[ 1 \longrightarrow N \xrightarrow{\ i \ } G \xrightarrow{\ p \ } Q \longrightarrow 1 \]
of an arbitrary group $Q$ by a group $N$ with \(\nu(N) = 0\). Then $\nu(G)\leq\nu(Q)$ with equality if the extension is split or central.
\end{proposition}
\begin{proof}
The inequality $\nu(G)\geq\nu(Q)$ for split extensions is clear because any section of \(p\) lifts any maximal free torus in \(Q\) to \(G\).

If the extension is central, the inequality $\nu(G)\geq\nu(Q)$ is deduced as follows. Let $(F_2)^n\leq Q$ be a free $n$-torus. Choose homomorphic sections $s_\ell\colon F_{2,\ell}\rightarrow G$ of $p$ for each factor $F_{2, \ell}$ of $(F_2)^n$, $1\leq \ell\leq n$.
Denote by $a_\ell, b_\ell$ a free basis of $s_\ell(F_{2, \ell})$. For $1\leq \ell\neq j\leq n$, consider the maps
\[
\begin{array}{llll}
\Phi_{\ell,j}\colon & s_j(F_{2, j})&\longrightarrow& i(N)\times i(N)\\
	&g & \longmapsto & ([a_\ell, g],[b_\ell, g]).
\end{array}
\]
Here and throughout the argument, we adopt the convention that $[x, y]=x^{-1}y^{-1}xy$.  By virtue of the commutator relation
\[[x, zy]=[x,y]y^{-1}[x, z]y\] 
and the fact that $i(N)$ is central and hence also abelian, the map $\Phi_{\ell, j}$ is a group homomorphism.  Hence we can consider the product homomorphism
\[ \Psi_j = \prod_{\ell \neq j} \Phi_{\ell,j} \colon \ s_j(F_{2,j}) \longrightarrow (i(N))^{2n-2}. \]
The kernel $K_j$ of \(\Psi_j\) is a subgroup of a free group, hence it is either trivial, infinite cyclic, or nonabelian free.  But it cannot be trivial because \((i(N))^{2n-2}\) is abelian.  Neither can it be infinite cyclic because nontrivial finitely generated normal subgroups of $F_2$ have finite index and \(F_2\) is not virtually cyclic.  Consequently $K_j$ is nonabelian free, hence contains a copy of $F_2$ which we want to call $L_j$.  By construction, \(L_1 \cdots L_n\) is a subgroup of \(G\) which is isomorphic to \((F_2)^n\).


Now we prove the inequality $\nu(G)\leq\nu(Q)$ for general extensions with \(\nu(N) = 0\).  Let $(F_2)^n\leq G$.  We will show that its image under $p$ is isomorphic to $(F_2)^n$.  For every copy $F_{2, j}\leq( F_2)^n$, $j\in\{1, \ldots, n\}$, consider the normal subgroup $H_j=F_{2, j}\cap i(N) = \ker (p|_{F_{2, j}})$ of \(F_{2,j}\).  Arguing as above, it cannot be infinite cyclic, hence it is trivial or nonabelian free.  But $i(N)$ has no nonabelian free subgroups by assumption, so $H_j$ is trivial. Consequently, the homomorphism $p|_{F_{2, j}}$ is injective.  This shows that the copy $F_{2, j}$ maps isomorphically to $p(F_{2,j})\leq Q$ for every $j\in\{1, \ldots, n\}$.

Obviously, the images $p(F_{2,j})$ of the different $F_2$-factors of $(F_2)^n$ commute in $Q$. What is more, each $p(F_{2,j})$ commutes with the subgroup generated by all other $p(F_{2, k})$ for $k\neq j$. Therefore we have
\[
p(F_{2,j})\cap\left\langle p(F_{2,k}), k\neq j\right\rangle=\{1\}.
\]
Indeed, if an element $q \in p(F_{2,j})$ also lies in $\left\langle p(F_{2,k}), k\neq j\right\rangle$, it must commute with all of $p(F_{2,j})$.  Hence $q$ must be the identity element.  This shows
\[
p((F_2)^n)=\left\langle p(F_{2,j}), j=1, \ldots, n\right\rangle\cong \prod_{j=1}^n p(F_{2,j})\cong (F_2)^n. \qedhere
\]
\end{proof}

\section{Semisimple subalgebras of semisimple Lie algebras}
\label{section:lie-algebras}

In this section and the next, all Lie algebras are understood to be complex Lie algebras.  We reserve the symbol \(\mathfrak{h}\) for a Cartan subalgebra of a semisimple Lie algebra \(\mathfrak{g}\) and we denote by \(\Phi(\mathfrak{g},\mathfrak{h}) \subset \mathfrak{h}^*\) the root system of \((\mathfrak{g},\mathfrak{h})\).  If the Cartan subalgebra is implicit or if its choice does not matter, we will also write \(\Phi(\mathfrak{g})\) or simply \(\Phi\) for short instead of \(\Phi(\mathfrak{g},\mathfrak{h})\).  For \(\alpha \in \Phi\), we denote the root space of \(\alpha\) by  \(\mathfrak{g}_\alpha \subset \mathfrak{g}\) and the coroot of \(\alpha\) (the Killing form dual) by \(h_\alpha \in \mathfrak{h}\).  The following definition, in an equivalent formulation, can be found in~\cite{Dynkin:subalgebras}*{Chapter\,II, \S\,5}.

\begin{definition}
  A subalgebra \(\mathfrak{f} \subset \mathfrak{g}\) of a semisimple Lie algebra is called \emph{regular} if for some root space decomposition \(\mathfrak{g} = \mathfrak{h} \oplus \bigoplus_{\alpha \in \Phi} \mathfrak{g}_\alpha\), we have
  \[ \mathfrak{f} = (\mathfrak{h} \cap \mathfrak{f}) \oplus \bigoplus_{\alpha \in \Phi} (\mathfrak{g}_\alpha \cap \mathfrak{f}). \]
\end{definition}

Hence the conjugacy class of a regular subalgebra is determined by a \emph{closed} subsystem \(\Sigma \subset \Phi\), meaning for \(\alpha, \beta \in \Sigma\) we have \(\alpha + \beta \in \Sigma\) whenever \(\alpha + \beta \in \Phi\), together with a subspace \(\mathfrak{h}_1 \subset \mathfrak{h}\) that contains the coroot \(h_\alpha\) whenever \(\alpha\) and \(-\alpha\) lie in \(\Sigma\).  Such a subalgebra is semisimple if and only if \(\alpha \in \Sigma\) implies \(-\alpha \in \Sigma\) and \(\mathfrak{h}_1\) is the span of the corresponding \(h_\alpha\).  In the latter case, \(\Sigma\) is itself a root sytem, namely the root system of \(\mathfrak{f}\) and \(\Sigma\) embeds as a \emph{closed subroot system} of \(\Phi\).  Even though a semisimple Lie algebra typically contains a multitude of non-regular semisimple subalgebras, we always have the following.

\begin{proposition} \label{prop:full-rank-regular}
  Let \(\mathfrak{f}\) be a semisimple subalgebra of a semisimple Lie algebra \(\mathfrak{g}\) such that \(\rank_\C \mathfrak{f} = \rank_\C \mathfrak{g}\).  Then \(\mathfrak{f}\) is regular.
\end{proposition}

For the convenience of the reader, we give an alternative argument to the original one in~\cite{Dynkin:subalgebras}*{Chapter\,II, \S\,6}.

\begin{proof}
By preservation of Jordan decomposition~\cite{Humphreys:Lie-algebras}*{Section~6.4, Exercise~9, p.\,31}, a Cartan subalgebra \(\mathfrak{h}\) of \(\mathfrak{f}\) is also a Cartan subalgebra of \(\mathfrak{g}\).  Hence there exists an element \(h \in \mathfrak{h}\) which is \emph{regular} in \(\mathfrak{g}\).  This means the endomorphism \(\operatorname{ad}(h)\) of \(\mathfrak{g}\) is diagonalizable with maximal possible spectrum so that the eigenspace decomposition gives a root space decomposition of \(\mathfrak{g}\).  Being diagonalizable means precisely that the minimal polynomial \(m\) of \(\operatorname{ad}(h)\) splits into simple roots.  Since \(m\) clearly also annihilates the restriction \(\operatorname{ad}(h)|_\mathfrak{f}\) to the invariant subspace \(\mathfrak{f} \subset \mathfrak{g}\), the minimal polynomial of \(\operatorname{ad}(h)|_\mathfrak{f}\) divides \(m\).  Therefore \(\operatorname{ad}(h)|_\mathfrak{f}\) is diagonalizable, too, and the eigenspaces are subspaces of the eigenspaces of \(\operatorname{ad}(h)\).  Whence every \(x \in \mathfrak{f}\) decomposes uniquely as a sum of \(\operatorname{ad}(h)\)-eigenvectors from \(\mathfrak{f}\).
\end{proof}

Adopting a terminology coined by E.\,B.\,Dynkin~\cite{Dynkin:subalgebras}*{\S 7}, we say more generally that a subalgebra of a semisimple Lie algebra \(\mathfrak{g}\) is an \emph{\(R\)-subalgebra} if it is contained in a \emph{proper} regular subalgebra of \(\mathfrak{g}\).  Any subalgebra of \(\mathfrak{g}\) which is not an \(R\)-subalgebra will be called an \emph{\(S\)-subalgebra}.  The following \emph{Kronecker product} algebras constitute a typical way in which \(S\)-subalgebras of semisimple Lie algebras can arise.

\begin{definition}[See \cite{Dynkin:subalgebras}, p. 238]\label{def:sums-of-Kronecker-products}
Let $\mathfrak{g}$ and $\mathfrak{k}$ be two classical matrix Lie algebras of size $s$ and \(t\), respectively.  We define the Lie algebra $\mathfrak{g}\times\mathfrak{k}$ to be the matrix Lie algebra of size \(st\) generated by all elements of the form
\[
g\otimes I_t+I_s\otimes k
\]
where $g\in \mathfrak{g}, k\in\mathfrak{k}$ and $I_\ell$ is the identity matrix of size $\ell$. Here $\otimes$ denotes the \emph{Kronecker product} of two matrices.
\end{definition}

While the factors \(\mathfrak{g}\) and \(\mathfrak{k}\) are somewhat strangely embedded into the matrix algebra \(\mathfrak{g} \times \mathfrak{k}\), the Kronecker product as abstract Lie algebra is still just a direct sum, as we verify next.

\begin{lemma}\label{lemma:Kronecker-algebras}
The algebra $\mathfrak{g}\times\mathfrak{k}$ is isomorphic to the direct sum $\mathfrak{g}\oplus\mathfrak{k}$.
\end{lemma}
\begin{proof}
The maps
\[
\begin{array}{lllllllll}
i_\mathfrak{g}\colon& \mathfrak{g} & \longrightarrow & \mathfrak{g}\times\mathfrak{k} & \mbox{and} &i_\mathfrak{k}\colon & \mathfrak{k}& \longrightarrow & \mathfrak{g}\times\mathfrak{k}\\
&g& \longmapsto &g\otimes I_t& && k&\longmapsto & I_s\otimes k 
\end{array}
\]
are injective linear maps and the union of their images spans $\mathfrak{g}\times\mathfrak{k}$ by definition.

The intersection $i_\mathfrak{g}(\mathfrak{g})\cap i_\mathfrak{k}(\mathfrak{k})$ is trivial: indeed, let   $A=g\otimes I_t=I_s\otimes k$ be one of its elements. As element of $i_\mathfrak{k}(\mathfrak{k})$, $A$ is a block diagonal matrix, the blocks being given by $k$. To be an element of $i_\mathfrak{g}(\mathfrak{g})$ implies then that $g$ is diagonal. But then $k$ is diagonal too. Moreover, the block structure implies that all the diagonal entries of $g$ must be the same. The same holds then for $k$. Now recall that the classical Lie algebras all have trace zero. Therefore the scalar matrices $g$ and $k$ are trivial.

It remains to check that the Lie bracket of $\mathfrak{g}\times\mathfrak{k}$ is compatible with the direct sum decomposition $i_\mathfrak{g}(\mathfrak{g})\oplus i_\mathfrak{k}(\mathfrak{k})$. Let $g\otimes I_t+I_s\otimes k, g'\otimes I_t+I_s\otimes k'\in\mathfrak{g}\times\mathfrak{k}$. Using the properties of the Kronecker product, we compute
\begin{align*}
[g\otimes I_t+I_s\otimes k, g'\otimes I_t+I_s\otimes k'] &= gg'\otimes I_t+g\otimes k'+g'\otimes k+I_s\otimes kk'\nonumber\\
	 & \quad - g'g\otimes I_t-g'\otimes k-g\otimes k'-I_s\otimes k'k\nonumber\\
	&= gg'\otimes I_t-g'g\otimes I_t+I_s\otimes kk'-I_s\otimes k'k\nonumber\\
	&= [g\otimes I_t, g'\otimes I_t]+[I_s\otimes k, I_s\otimes k']\nonumber\\
	& = [i_\mathfrak{g}(g), i_\mathfrak{g}(g')]+[i_\mathfrak{k}(k), i_\mathfrak{k}(k')].\nonumber \qedhere
\end{align*}
\end{proof}

\section{Maximal products of minimal subalgebras}
\label{section:maximal}

In this section, we exploit Dynkin's classical work on the classification of semisimple subalgebras of semisimple Lie algebras to give a proof of the following proposition.  We still assume that all Lie algebras are complex.

\begin{proposition}\label{prop:regular-versus-not}
Let \(\mathfrak{g}\) be a semisimple Lie algebra.  Then the maximal number \(n\) such that \((\mathfrak{sl}_2)^n\) embeds as a regular subalgebra of \(\mathfrak{g}\) is also maximal among all embeddings of \((\mathfrak{sl}_2)^n\) as subalgebra.
\end{proposition}

We give the proof by induction on \(\dim \mathfrak{g}\).  The beginning is trivial because \(\dim \mathfrak{g} = 3\) implies \(\mathfrak{g} \cong \mathfrak{sl}_2\).  So consider a subalgebra \((\mathfrak{sl}_2)^n \subset \mathfrak{g}\) with maximal~\(n\).

\subsection{Reduction to embedding  \((\mathfrak{sl}_2)^n\) as $S$-subalgebra in a simple \(\mathfrak{g}\) }

If \((\mathfrak{sl}_2)^n \subset \mathfrak{g}\) is an \(R\)-subalgebra, then \cite{Dynkin:subalgebras}*{Theorem~7.7} implies that \((\mathfrak{sl}_2)^n\) is contained in a proper regular semisimple subalgebra \(\mathfrak{f} \subset \mathfrak{g}\).  By induction assumption, there exists a (possibly different) regular subalgebra \((\mathfrak{sl}_2)^n \subset \mathfrak{f}\) which is also regular in \(\mathfrak{g}\) because being regular is apparently transitive.  Hence it remains to consider the case that \((\mathfrak{sl}_2)^n\) is an \(S\)-subalgebra of~\(\mathfrak{g}\).  Assume \(\mathfrak{g} = \mathfrak{g}_1 \oplus \mathfrak{g}_2\) for two semisimple ideals \(\mathfrak{g}_1\) and \(\mathfrak{g}_2\).  Then \((\mathfrak{sl}_2)^n\) projects to subalgebras \((\mathfrak{sl}_2)^{n_1}\) of \(\mathfrak{g}_1\) and \((\mathfrak{sl}_2)^{n_2}\) of \(\mathfrak{g}_2\) with \(n_1 + n_2 = n\).  By induction assumption, we have regular subalgebras \((\mathfrak{sl}_2)^{n_1}\) in \(\mathfrak{g}_1\) and \((\mathfrak{sl}_2)^{n_2}\) in \(\mathfrak{g}_2\) and their direct sum gives a regular subalgebra \((\mathfrak{sl}_2)^n\) in \(\mathfrak{g}\). This reduces the proof to the case that \(\mathfrak{g}\) is simple.

\subsection{Maximal regular embeddings of \((\mathfrak{sl}_2)^n\) in maximal $S$-subalgebras}

Suppose that \((\mathfrak{sl}_2)^n \subset \mathfrak{g}\) is embedded as an \(S\)-subalgebra and \(\mathfrak{g}\) is simple.  Since \(\dim \mathfrak{g} > 3\), we know that \((\mathfrak{sl}_2)^n \subset \mathfrak{g}\) is a proper subalgebra, hence lies in a maximal proper \(S\)-subalgebra \(\mathfrak{g}_1\) of the simple algebra \(\mathfrak{g}\).  By induction assumption, \((\mathfrak{sl}_2)^n\) also embeds regularly in \(\mathfrak{g}_1\) and it remains to see that this implies that \(n\) is less than or equal to the maximal \(m\) such that \((\mathfrak{sl}_2)^m\) embeds regularly into \(\mathfrak{g}\).

\subsubsection{Maximal \(S\)-subalgebras \(\mathfrak{g}_1\) of exceptional simple Lie algebras \(\mathfrak{g}\)}

From \cite{Dynkin:subalgebras}*{Theorem~14.1 and Tables~9 and~11}, we extract to Table~\ref{table:exceptional} which types of maximal proper \(S\)-subalgebras occur among the exceptional simple Lie algebras yielding the indicated maximal values of \(m\) and \(n\).
  \begin{table}
  \begin{center}
    \renewcommand{\arraystretch}{1.5}
  \begin{tabular}{c l c c}
    \hline
    \textbf{Type} & \textbf{Types of maximal proper \textit{S}-subalgebras} & \textbf{\textit{m}} & \textbf{\textit{n}}\\
    \hline
    \(G_2\) \ & \(A_1\) & 2 & 1 \\
    \(F_4\) \ & \(A_1\), \ \ \(G_2 \times A_1\) & 4 & 3\\
    \(E_6\) \ & \(A_1\), \ \ \(G_2\), \ \ \(C_4\), \ \ \(G_2 \times A_2\), \ \ \(F_4\) & 4 & 4\\
    \(E_7\) \ & \(A_1\), \ \ \(A_2\), \ \ \(G_2 \times C_3\), \ \ \(F_4 \times A_1\), \ \ \(G_2 \times A_1\), \ \ \(A_1 \times A_1\) \ \ & 7 & 5\\
    \(E_8\) \ & \(A_1\), \ \ \(G_2 \times F_4\), \ \ \(A_2 \times A_1\), \ \ \(B_2\) & 8 & 6\\
    \hline
    \end{tabular}
  \end{center}
  \caption{Maximal proper \(S\)-subalgebras of exceptional algebras.}
  \label{table:exceptional}
\end{table}
We read off immediately that for exceptional \(\mathfrak{g}\), we always have \(m \ge n\), as required.

\subsubsection{Maximal \(S\)-subalgebras \(\mathfrak{g}_1\) of classical simple Lie algebras \(\mathfrak{g}\)}

It remains to consider the case that \((\mathfrak{sl}_2)^n\) lies regularly in a maximal proper \(S\)-subalgebra \(\mathfrak{g}_1\) of a classical simple algebra~\(\mathfrak{g}\).  To settle this case, we can make usage of Dynkin's earlier work~\cite{Dynkin:maximal-subgroups} in which he classifies \emph{all} maximal subalgebras of the classical algebras, because the maximal \(S\)-subalgebra \(\mathfrak{g}_1\) is by definition of an \(S\)-subalgebra also maximal among all subalgebras.

The classification of maximal subalgebras of classical simple Lie algebras splits into three categories~\cite{Dynkin:subalgebras}*{No.\,46, p.\,237}:
  \begin{enumerate}[(I)]
  \item regular,
  \item non-regular simple,
  \item non-regular non-simple.
  \end{enumerate}
Since \(\mathfrak{g}_1\) is an \(S\)-subalgebra, it must either belong to category~II or category~III. 

The types of the {\bfseries category III} maximal subalgebras of the classical simple algebras can be inferred from \cite{Dynkin:subalgebras}*{Table on p.\,238} so that we compute the values for the maximal \(m\) and \(n\) given in Table~\ref{table:category3} by means of \cite{Dynkin:subalgebras}*{Table~9}.
 \begin{table} 
    \begin{center}
    \renewcommand{\arraystretch}{2}
  \begin{tabular}{llll}
    \hline
    \textbf{Type} & \textbf{Types of category III subalgebras} & \textbf{\textit{m}} & \textbf{\textit{n}}\\
    \hline
    \(A_r\) \ & \small \(A_{s-1} \times A_{t-1} \ \text{ for } \ 2 \le s \le t, \ \ st = r+1\) & \(\lceil \frac{r}{2} \rceil\) & \(\lfloor \frac{s}{2} \rfloor + \lfloor \frac{t}{2} \rfloor\) \\
    \(B_r\) & \small \(B_s \times B_t \ \text{ for } 1 \le s \le t, \ (2s+1)(2t+1) = 2r+1\) & \(r\) & \(s + t\) \vspace{1mm}\\
    \(C_r\) & \small \makecell[l]{\(C_s \times B_t \ \text{ for } \ s \ge 1, \ t \ge 1, \ \ s(2t+1) = r\), \\ \(C_s \times D_t \ \text{ for } \ s \ge 1, \ t \ge 3, \ \ s(2t) = r\) \\ \(C_1 \times D_2, \ r = 4\)} & \makecell[l]{\(r\) \\ \(r\) \\ \(4\)} & \makecell[l]{\(s+t\) \\ \(\le s+t\) \\ 3} \vspace{2mm} \\

    \(D_r\) &\small \makecell[l]{\(C_s \times C_t \ \, \text{ for } \ 1 \le s \le t, \ \ 2st=r\), \\  \(B_s \times D_t \ \text{ for } \ 1 \le s < t, \ (2s+1)t=r, \ \ t \neq 2\), \\ \(D_s \times B_t \ \text{ for } \ 2 < s < t + 1, \ \ s(2t+1) = r\), \\ \(D_s \times D_t \ \text{ for } \ 2 < s \le t, \ \ 2st=r\)} &  \small \makecell[l]{\(r\) if \(r\) is  \\ even, \\ \(r-1\) if \(r\) \\ is odd } & \(\le s+t\)\\
    \hline
    \end{tabular}
  \end{center}
  \caption{Category~III subalgebras of classical simple Lie algebras.}
  \label{table:category3}
\end{table}
 It turns out that as matrix algebras, the category~III subalgebras lie inside \(\mathfrak{g}\) as sums of \emph{Kronecker products}, see Definition \ref{def:sums-of-Kronecker-products} and Lemma \ref{lemma:Kronecker-algebras}.  This also explains the conditions of the type ``\(st = r+1\)'' resulting from the embedding \(\mathfrak{sl}_s \times \mathfrak{sl}_t \subset \mathfrak{sl}_{st}\).  The term ``\(\le s + t\)'' in the \(n\)-column means more precisely ``\(s + t -1\)'' if one of type \(D_t\) with \(t\) odd or \(D_s\) with \(s\) odd occur as factors in the maximal subalgebra, ``\(s + t -2\)'' if both \(D_t\) with \(t\) odd and \(D_s\) with \(s\) odd occur as factors in the maximal subalgebra, and ``\(s+t\)'' in all other cases.  As a side remark, we observe that if \(2r+1\) is prime, then a type \(B_r\) algebra has no maximal subalgebra which is neither regular nor simple.  Of course, the usual low rank isomorphisms can and should be applied; for example \(C_4\) has a maximal subalgebra which is non-regular and of type \(C_1 \times D_2 \cong A_1 \times A_1 \times A_1\).  The table shows again that \(m \ge n\) (and equality only occurs if \(\mathfrak{g}\) has type~\(A_3\) or \(D_2\)).

So the final case to consider occurs if \(\mathfrak{g}_1\) belongs to  {\bfseries category~II}, meaning \(\mathfrak{g}_1\) is a non-regular simple maximal subalgebra of either \(\mathfrak{sl}_k\), \(\mathfrak{sp}_k\), or \(\mathfrak{so}_k\).

This case typically arises as follows.  Take a \(k\)-dimensional irreducible representation \(\phi \colon \mathfrak{g}_1 \rightarrow \mathfrak{gl}_k\) of the simple Lie algebra \(\mathfrak{g}_1\).  Since every simple Lie algebra is perfect, we have
  \[ \phi(\mathfrak{g}_1) = \phi([\mathfrak{g}_1,\mathfrak{g}_1]) = [\phi(\mathfrak{g}_1), \phi(\mathfrak{g}_1)] \subset [\mathfrak{gl}_k, \mathfrak{gl}_k] = \mathfrak{sl}_k, \]
  so \(\phi\) defines an embedding \(\mathfrak{g}_1 \subset \mathfrak{sl}_k\) assuming \(k > 1\).  If the corresponding linear group \(G_0 \subset \operatorname{SL}_k(\C)\) preserves a non-degenerate skew-symmetric bilinear form, we obtain in fact \(\mathfrak{g}_1 \subset \mathfrak{sp}_k\) and \(k\) must be even.  If \(G_0\) preserves a non-degenerate symmetric bilinear form, we obtain \(\mathfrak{g}_1 \subset \mathfrak{so}_k\).  Whether one of the last two (mutually exclusive) conditions is satisfied, can be decided solely in terms of the integers \(\frac{2(\lambda, \alpha)}{(\alpha, \alpha)}\) for simple roots \(\alpha\) of \(\mathfrak{g}_1\) where \(\lambda\) is the highest weight of \(\phi\)~\cite{Dynkin:maximal-subgroups}*{Paragraph~C, p.\,254}.

 It turns out that these subalgebras \(\mathfrak{g}_1 \subset \mathfrak{sl}_k\), \(\mathfrak{g}_1 \subset \mathfrak{sp}_k\), and \(\mathfrak{g}_1 \subset \mathfrak{so}_k\) are almost always maximal.  In fact, the irreducible representations \(\phi \colon \mathfrak{g}_1 \rightarrow \mathfrak{gl}_k\) which form an exception to this rule either belong to one of four infinite series or one of 14 additional cases which are all listed in \cite{Dynkin:maximal-subgroups}*{Table~1, p.\,364}.  For our purposes it is however more important that according to \cite{Dynkin:subalgebras}*{p.\,238}, with the exception of \(\mathfrak{so}_{2r-1} \subset \mathfrak{so}_{2r}\), every category II maximal subalgebra \(\mathfrak{g}_1\) of a classical simple Lie algebra \(\mathfrak{g}\) arises in the above fashion.

For \(\mathfrak{so}_{2r-1} \subset \mathfrak{so}_{2r}\), we have \(n = r-1\) and \(m=r\) for even \(r\) and \(m=r-1\) for odd \(r\), so in any case \(m \ge n\).  Thus it is now enough to see that for proper subalgebras \(\mathfrak{g}_1 \subset \mathfrak{sl}_k\), \(\mathfrak{g}_1 \subset \mathfrak{sp}_k\), and \(\mathfrak{g}_1 \subset \mathfrak{so}_k\) constructed as above, we still have \(m \ge n\).

It is clearly enough to verify this condition for the minimal possible \(k\) in each case.  The highest weight of a minimal faithful representation of \(\mathfrak{g}_1\) must occur among the fundamental weights so that we can look up these dimensions in the relevant literature, for instance~\cite{Carter:Lie-Algebras}*{Propositions~13.2, 13.8, 13.10, 13.23, and Section~13.8}, and gather them in Table~\ref{table:category2}.
  \begin{table}
  \begin{center}
  \begin{tabular}{lccccccccc}
    \hline
    \textbf{Type of} \(\mathfrak{g}_1\) & \(A_r\) & \(B_r\) & \(C_r\) & \(D_r\) & \(E_6\) & \(E_7\) & \(E_8\) & \(F_4\) & \(G_2\) \\
    \textbf{minimal dim} & \(r+1\) & \(2r+1\) & \(2r\) & \(2r\) & \(27\) & \(56\) & \(248\) & \(26\) & \(7\) \\
    \hline                                  
  \end{tabular}
\end{center}
\caption{Minimal dimensions of faithful representations of simple Lie algebras.}
\label{table:category2}
\end{table}
For the classical types, the following restrictions apply: \(r \ge 1\) for \(A_r\), \(r \ge 3\) for \(B_r\), \(r \ge 2\) for \(C_r\), and \(r \ge 4\) for \(D_r\).

In the classical cases, the minimal faithful representations are (up to outer automorphisms of \(\mathfrak{g}_1\)) the standard vector representations so that these yield no proper embeddings of \(\mathfrak{g}_1\) in \(\mathfrak{g} = \mathfrak{sl}_k\), \(\mathfrak{g} = \mathfrak{sp}_k\), or \(\mathfrak{g} = \mathfrak{so}_k\).  Therefore, the minimal possible \(k\) for classical \(\mathfrak{g}_1\) is strictly greater than the dimensions given in the table.  In the three cases \(\mathfrak{g} = \mathfrak{sl}_k\), \(\mathfrak{g} = \mathfrak{sp}_k\), and \(\mathfrak{g} = \mathfrak{so}_k\) with \(k \neq 2 \text{ mod } 4\), we have \(m = \lfloor \frac{k}{2} \rfloor\) as we conclude again from \cite{Dynkin:subalgebras}*{Table~9} or from our Table~\ref{table:category3}.  For \(\mathfrak{g} = \mathfrak{so}_k\) with \(k = 2 \text{ mod } 4\), we have in turn \(m = \frac{k}{2}-1\).  Using these formulas, one easily checks from Table~\ref{table:category2} that always \(m \ge n\).  With this last case, the proof of Proposition~\ref{prop:regular-versus-not} is complete.

\section{Free subgroup rank of simple Lie groups}
\label{section:simple}

Now we are all set for proving Theorem~\ref{thm:nu-of-simple}.  So let \(G\) be a connected simple Lie group and let \(\mathfrak{g}_0\) be the Lie algebra of \(G\).  For the sake of clarity, we split the various parts of the proof into different subsections.

\subsection{Proof of part~\eqref{item:complex-structure}: assuming \(\mathfrak{g}_0\) has a complex structure.} We first treat the non-absolutely simple case when the real Lie algebra \(\mathfrak{g}_0\) possesses a complex structure.   To show the inequality \(\nu(G) \le \operatorname{sork} \Phi (\mathfrak{g}_0)\), let us fix such a complex structure on \(\mathfrak{g}_0\) and let \((F_2)^n \le G\) be a free \(n\)-torus.  Since \((F_2)^n\) is a center-free group, it must intersect the center \(Z(G)\) of the Lie group \(G\) in the trivial subgroup.  This has the effect that \((F_2)^n\) still embeds into \(G / Z(G)\).  The adjoint representation identifies \(G / Z(G)\) with the unit component \(\mathbf{G}^0\) of the group of complex points of the linear algebraic \(\C\)-group \(\mathbf{G} = \operatorname{Aut}_\C(\mathfrak{g}_0)\).  Let \(\mathbf{G}_j\) be the Zariski closure in \(\mathbf{G}\) of the \(j\)-th factor \(F_{2,j}\leq (F_2)^n\). 
Then \(\mathbf{G}_j\) contains a noncommutative free subgroup, hence is not solvable.  Therefore the corresponding Lie subalgebra \(\mathfrak{g}_j \subset \mathfrak{g}_0\) contains a nontrivial semisimple Levi subalgebra \(\mathfrak{s}_j \subset \mathfrak{g}_j\) which in turn contains a subalgebra \(\mathfrak{sl}_{2, j} \subset \mathfrak{s}_j\) isomorphic to the standard \(\mathfrak{sl}_2(\C)\).  Since the \(j\)-th and the \(k\)-th factor of \((F_2)^n\) commute for \(j \neq k\), so do \(\mathbf{G}_j\) and \(\mathbf{G}_k\) by density.  We conclude \([\mathfrak{sl}_{2, j}, \mathfrak{sl}_{2, k}] = 0\) for \(j \neq k\) because the Lie bracket \([X,Y]\) of two left invariant vector fields \(X\) and \(Y\) of \(\mathbf{G}\) is by definition the infinitesimal commutator \(\frac{\textup{d}}{\textup{d}t}\vert_{t=0} \big(\Phi^Y_{-\sqrt{t}} \circ \Phi^X_{-\sqrt{t}} \circ \Phi^Y_{\sqrt{t}} \circ \Phi^X_{\sqrt{t}}\big)\) of the flows \(\Phi^X\) and \(\Phi^Y\) of \(X\) and \(Y\).  As \(\mathfrak{sl}_2(\C)\) is center-free, this also implies \(\mathfrak{sl}_{2,j} \cap \mathfrak{sl}_{2,k} = \{0\}\) for \(j \neq k\).  Thus we have found a semisimple subalgebra of the form \((\mathfrak{sl}_2)^n \subset \mathfrak{g}_0\).  By Proposition~\ref{prop:regular-versus-not}, we also find a regular subalgebra \((\mathfrak{sl}_2)^n \subset \mathfrak{g}_0\) with respect to some Cartan subalgebra \(\mathfrak{h} \subset \mathfrak{g}_0\) and hence a closed subroot system \(\Sigma \subset \Phi(\mathfrak{g}, \mathfrak{h})\) of type \((A_1)^n\).  Picking one root from each \(A_1\)-factor yields a set of \(n\) pairwise strongly orthogonal roots, showing \(n \le \operatorname{sork} \Phi(\mathfrak{g}_0)\).

Conversely, unifying a set of \(n\) pairwise strongly orthogonal roots in \(\Phi(\mathfrak{g}_0)\) with their opposite roots yields a closed subroot system \(\Sigma \subset \Phi(\mathfrak{g}_0)\) of type \((A_1)^n\) and hence a regular subalgebra \((\mathfrak{sl}_2)^n \subset \mathfrak{g}_0\).  Since \(G\) is assumed connected and has a complex Lie algebra, it has the structure of a complex Lie group.  By the above, it contains a subgroup finitely covered by \((\operatorname{SL}_2(\C))^n\).  Example~\ref{example:nrank}\,\eqref{item:sl2z}, Proposition~\ref{prop:products}\,\eqref{item:direct-product}, and the argument involving the center from above show that \(n \le \nu(G)\).

\subsection{Proof of part~\eqref{item:real-form}: assuming \(\mathfrak{g}_0\) has no complex structure}

Now we treat the absolutely simple case and assume that \(\mathfrak{g}_0\) does not admit any complex structure. By part~\eqref{item:complex-structure}, we have \(\nu(G) \le \nu(G_\C) = \operatorname{sork} \Phi(\mathfrak{g})\) where \(G_\C\) denotes the complexification of the Lie group \(G\).  To see that part~\eqref{item:complex-structure} actually applies to \(G_\C\), one may argue that \(G_\C\) is simple because \(\mathfrak{g}_0\) admits no complex structure.  Thus \(G_\C\) is also simple as real Lie group.  Alternatively, one realizes that the above proof of~\eqref{item:complex-structure} actually works for semisimple groups and algebras.

So in this section, we only need to show the inequality \(\nu(G) \ge \operatorname{sork} \Phi(\mathfrak{g})\) under the assumption that \(\mathfrak{g}_0\) is not isomorphic to \(\mathfrak{so}(p,q)\) with odd \(p\) and \(q\) and \(p+q\) divisible by four.  We start by introducing some terminology and notation.

\subsubsection{Preliminaries}

We denote the Killing form of \(\mathfrak{g}_0\) by \(B_0(x,y) = \operatorname{tr}(\operatorname{ad} x \circ \operatorname{ad} y)\) and we pick a \emph{Cartan involution} of \(\mathfrak{g}_0\), meaning an order two automorphism \(\theta_0\) of \(\mathfrak{g}_0\) such that the form \(B_{\theta_0} (x,y) = -B_0(x,\theta_0 (y))\) is positive definite.  It is then a simple computation that ``adjoints of adjoints'' with respect to \(B_{\theta_0}\) are given by \((\operatorname{ad}(x))^* = -\operatorname{ad}(\theta_0 (x))\).  The \({+1}/{-1}\) eigenspace decomposition \(\mathfrak{g}_0 = \mathfrak{k}_0 \oplus \mathfrak{p}_0\) of \(\theta_0\) is called a \emph{Cartan decomposition} of~\(\mathfrak{g}_0\).  While \(\mathfrak{k}_0\) is a subalgebra, \(\mathfrak{p}_0\) is a subspace satisfying \([\mathfrak{p}_0,\mathfrak{p}_0] \subset \mathfrak{k}_0\).  Finally, let \(\mathfrak{h}_0 \subset \mathfrak{g}_0\) be a \(\theta_0\)-stable Cartan subalgebra.  We drop the index ``\(0\)'' when referring to the complexifications of the objects introduced so far: the simple complex Lie algebra \(\mathfrak{g} = \mathfrak{g}_0 \otimes \C\) with Cartan subalgebra \(\mathfrak{h} = \mathfrak{h}_0 \otimes \C\), the Cartan involution \(\theta = \theta_0 \otimes \C\), the symmetric \(\C\)-bilinear Killing form \(B = B_0 \otimes \C\), and the hermitian form \(B_\theta\) defined as the unique sesquilinear extension of \(B_{\theta_0}\) from \(\mathfrak{g}_0\) to \(\mathfrak{g}\).

Both \(\theta\) and the complex conjugation \(\sigma\) of \(\mathfrak{g}\) with respect to \(\mathfrak{g}_0\) act as involutions on the root system \(\Phi(\mathfrak{g},\mathfrak{h})\).  To see this, we compute that for every \(\alpha \in \Phi(\mathfrak{g},\mathfrak{h})\), every \(h \in \mathfrak{h}\) and every \(x \in \mathfrak{g}_\alpha\), we have
\[ [h, \theta(x)] = \theta([\theta(h), x]) = \theta(\alpha(\theta(h))x) = (\alpha \circ \theta)(h) \,\theta(x). \]
Setting \(\alpha^{\theta} = \alpha \circ \theta\), we thus have \(\theta(\mathfrak{g}_\alpha) = \mathfrak{g}_{\alpha^\theta}\).  Similarly, we obtain \(\sigma(\mathfrak{g}_\alpha) = \mathfrak{g}_{\alpha^\sigma}\) when setting \(\alpha^\sigma = \overline{\alpha \circ \sigma}\).  

\subsubsection{The split case}

Let us now assume that \(\mathfrak{g}_0\) is a \emph{split real form} of~\(\mathfrak{g}\), meaning the \(\theta_0\)-stable Cartan subalgebra \(\mathfrak{h}_0\) can be chosen within \(\mathfrak{p}_0\).  With such \(\mathfrak{h}_0 \subset \mathfrak{p}_0\), every \(h \in \mathfrak{h}_0\) defines a self-adjoint endomorphism \(\operatorname{ad}_{\mathfrak{g}_0}(h)\) of the Euclidean space \((\mathfrak{g}_0, B_{\theta_0})\).  The endomorphism \(\operatorname{ad}_{\mathfrak{g}}(h)\) of the unitary space \((\mathfrak{g}, B_\theta)\) is thus also self-adjoint, being the complexification of \(\operatorname{ad}_{\mathfrak{g}_0}(h)\).  This shows that all roots \(\alpha \in \mathfrak{h}^*\) of \(\mathfrak{g}\) take real values on \(\mathfrak{h}_0\).  Decomposing \(h \in \mathfrak{h} = \mathfrak{h}_0 \oplus \textup{i} \mathfrak{h}_0\) as \(h = h_1 + \textup{i} h_2\), we thus have
\[ \overline{\alpha(\sigma(h))} = \overline{\alpha(h_1 - \textup{i} h_2)} = \overline{\alpha(h_1)-\textup{i}\alpha(h_2)} = \alpha(h_1) + \textup{i} \alpha(h_2) = \alpha(h), \]
so \(\alpha^\sigma = \alpha\).  Hence in case \(\mathfrak{g}_0\) is a split form, \(\sigma\) acts trivially on \(\Phi(\mathfrak{g},\mathfrak{h})\) and preserves the root space decomposition of \(\mathfrak{g}\) with respect to~\(\mathfrak{h}\).  Therefore, by the exact same argument as in~\eqref{item:complex-structure}, a set of \(n\) pairwise strongly orthogonal roots yields a subalgebra \((\mathfrak{sl}_2(\R))^n \subset \mathfrak{g}_0\).  Thus there exists a subgroup in \(G\) which is isogenous to \((\operatorname{SL}_2(\R))^n\), whence \((F_2)^n \le G\) and \(\operatorname{sork} \Phi(\mathfrak{g}) \le \nu(G)\).

\subsubsection{The case of a compact Cartan subalgebra} \label{sec:compact-cartan}

Next we consider the opposite case and assume that we can find a Cartan subalgebra \(\mathfrak{h}_0 \subset \mathfrak{k}_0\).  In addition to \(\mathfrak{g}_0\), we consider the real subalgebra \(\mathfrak{u}_0 = \mathfrak{k}_0 \oplus \textup{i} \mathfrak{p}_0 \subset \mathfrak{g}\).  It is called a \emph{compact real form} of \(\mathfrak{g}\) because \(\mathfrak{g} = \mathfrak{u}_0 \otimes \C\) and because the corresponding real Lie group \(U \subset G_\C\) is compact since the Killing form of \(\mathfrak{u}_0\) is negative definite, being the restriction of \(B\) to \(\mathfrak{u}_0\).  Let \(\tau\) be the complex conjugation of \(\mathfrak{g}\) with respect to~\(\mathfrak{u}_0\).  Both \(\sigma\) and \(\tau\) are complex antilinear, Lie bracket preserving, commuting involutions of \(\mathfrak{g}\) and their composition in any order equals \(\theta\).  According to~\cite{Knapp:beyond}*{Theorem~6.6, p.\,295}, we can pick \(x_\alpha \in \mathfrak{g}_\alpha\) for each \(\alpha \in \Phi(\mathfrak{g}, \mathfrak{h})\) such that
\begin{align*}
[x_\alpha, x_{-\alpha}] &= h_\alpha & \\
[x_\alpha, x_\beta] &= n_{\alpha, \beta} x_{\alpha + \beta}  &\text{ if } \alpha + \beta \in \Phi(\mathfrak{g},\mathfrak{h}) \\
[x_\alpha, x_\beta] &= 0 &\text{ if } \alpha + \beta \neq 0 \text{ and } \alpha + \beta \notin \Phi(\mathfrak{g}, \mathfrak{h})
\end{align*}
where the constants \(n_{\alpha, \beta}\) satisfy \(n_{\alpha, \beta} = -n_{-\alpha, -\beta}\).  The union of the \(x_\alpha\) with the coroots \(h_i = h_{\alpha_i}\) of a choice of simple roots \(\alpha_i \in \Delta \subset \Phi(\mathfrak{g},\mathfrak{h})\) is sometimes known as a \emph{Chevalley basis} of \(\mathfrak{g}\).  It is well-known that all structure constants of \(\mathfrak{g}\) with respect to a Chevalley basis are integers~\cite{Humphreys:Lie-algebras}*{Theorem~25.2, p.\,147}.  By~\cite{Knapp:beyond}*{Theorem~6.11 and~(6.12), p.\,297}, the \(\R\)-span
\[ \langle \,\textup{i} h_\alpha, \,x_\alpha - x_{-\alpha}, \,\textup{i}(x_\alpha + x_{-\alpha}) \,\colon\, \alpha \in \Phi(\mathfrak{g},\mathfrak{h}) \,\rangle_\R \]
is a compact real form of \(\mathfrak{g}\).  Since all compact real forms of \(\mathfrak{g}\) are conjugate by an inner automorphism of \(\mathfrak{g}\) \cite{Knapp:beyond}*{Corollary~6.20, p.\,302}, we may assume that this span is \(\mathfrak{u}_0\) to begin with.  It is then immediate that \(\tau(x_\alpha) = - x_{-\alpha}\), hence \(\tau(h_\alpha) = \tau([x_\alpha, x_{-\alpha}]) = -h_{\alpha} = h_{-\alpha}\), so \(\tau\) acts on the root system \(\Phi(\mathfrak{g},\mathfrak{h})\) by point reflection.  Since \(\theta\) restricts to the identity on \(\mathfrak{h}_0 \subset \mathfrak{k}_0\), we have \(\alpha^\theta = \alpha\), so \(\theta\) acts identically on \(\Phi(\mathfrak{g},\mathfrak{h})\).  Therefore \(\sigma\), just like \(\tau\), acts by point reflection.  Since \(\theta(\mathfrak{g}_\alpha) = \mathfrak{g}_\alpha\) for all roots \(\alpha \in \Phi = \Phi(\mathfrak{g},\mathfrak{h})\), we must have \(\mathfrak{g}_\alpha \subset \mathfrak{k}\) or \(\mathfrak{g}_\alpha \subset \mathfrak{p}\) because root spaces are one-dimensional.

\begin{definition}
We call \(\alpha \in \Phi\) \emph{compact} if \(\mathfrak{g}_\alpha \subset \mathfrak{k}\) and \emph{noncompact} if \(\mathfrak{g}_\alpha \subset \mathfrak{p}\).
\end{definition}

Hence compact roots \(\alpha\) satisfy \(\theta (x_\alpha) = x_\alpha\) and \(\sigma(x_\alpha) = -x_{-\alpha}\) whereas noncompact roots satisfy \(\theta(x_\alpha) = -x_\alpha\) and \(\sigma(x_\alpha) = x_{-\alpha}\).  Moreover, for \(\theta(x_\alpha) = \pm x_\alpha\) we compute
\[ \theta(x_{-\alpha}) = \theta(-\tau(x_\alpha)) = -\tau(\theta(x_\alpha)) = -\tau(\pm x_\alpha) = -(\mp x_{-\alpha}) = \pm x_{-\alpha}, \]
hence a root is always of the same type as the opposite root.

For every root \(\alpha \in \Phi(\mathfrak{g}, \mathfrak{h})\), the closed subroot system \(\{ \pm \alpha \} \subset \Phi(\mathfrak{g}, \mathfrak{h})\) determines the regular subalgebra \(\mathfrak{sl}_\alpha = \langle h_\alpha, x_\alpha, x_{-\alpha} \rangle_\C \subset \mathfrak{g}\) isomorphic to \(\mathfrak{sl}_2(\C)\).  With the above formulas, we see that always \(\sigma(\mathfrak{sl}_\alpha) = \mathfrak{sl}_\alpha\) and that
\[ \langle \,\textup{i} h_\alpha, \,x_\alpha \mp x_{-\alpha}, \,\textup{i}(x_\alpha \pm x_{-\alpha}) \,\rangle_\R, \]
is the algebra \(\mathfrak{sl}_\alpha^\sigma\) of \(\sigma\)-fixed points of \(\mathfrak{sl}_\alpha\), where the upper signs apply for compact \(\alpha\) and the lower signs apply for noncompact \(\alpha\).  Since \(\mathfrak{sl}_\alpha^\sigma\) is a real form of the simple algebra \(\mathfrak{sl}_\alpha\), it must be simple itself, hence it is either isomorphic to \(\mathfrak{su}(2)\) or to \(\mathfrak{sl}_2(\R)\).  If \(\alpha\) is compact, we have that 
\[ \mathfrak{sl}_\alpha^\sigma = \mathfrak{sl}_\alpha^\tau \subset \mathfrak{u}_0 \cap \mathfrak{g}_0 = \mathfrak{k}_0 \]
is compact, hence \(\mathfrak{sl}_\alpha^\sigma \cong \mathfrak{su}(2)\).  If \(\alpha\) is noncompact, the endomorphism \(\operatorname{ad}_{\mathfrak{sl}_\alpha^\sigma}(x_\alpha + x_{-\alpha})\) has the real transformation matrix
\[ \begin{pmatrix} 0 & 0 & -2 \\ 0 & 0 & 0 \\ -\alpha(h_\alpha) & 0 & 0 \end{pmatrix} \]
with respect to the above basis.  So inserting \(x_\alpha + x_{-\alpha}\) in both arguments of the Killing form of \(\mathfrak{sl}_\alpha^\sigma\) gives
\begin{align*}
& \operatorname{tr}(\operatorname{ad}_{\mathfrak{sl}_\alpha^\sigma}(x_\alpha + x_{-\alpha})^2) = 4 \alpha(h_\alpha) = 4 B(h_\alpha, h_\alpha) = -4 B_0(\textup{i} h_\alpha, \theta_0(\textup{i} h_\alpha)) = \\
& = 4B_{\theta_0}(\textup{i}h_\alpha, \textup{i}h_\alpha) > 0.
\end{align*}
Thus \(\mathfrak{sl}_\alpha^\sigma \subset \mathfrak{g}_0\) is noncompact in this case and must be isomorphic to \(\mathfrak{sl}_2(\R)\).  Since \(\mathfrak{sl}_\alpha\) and \(\mathfrak{sl}_\beta\) commute for a pair of strongly orthogonal roots \(\alpha, \beta \in \Phi(\mathfrak{g},\mathfrak{h})\), so do the real subalgebras \(\mathfrak{sl}_\alpha^\sigma\) and \(\mathfrak{sl}_\beta^\sigma\).

Given a set \(\Omega \subset \Phi(\mathfrak{g},\mathfrak{h})\) of \(n\) pairwise strongly orthogonal roots, we have therefore found a subalgebra
\[ \mathfrak{f}_0 = \bigoplus_{\alpha \in \Omega} \mathfrak{sl}_\alpha^\sigma \subset \mathfrak{g}_0 \]
with \(\mathfrak{sl}_\alpha^\sigma\) isomorphic to \(\mathfrak{su}(2)\) or \(\mathfrak{sl}_2(\R)\) depending on whether \(\alpha\) is compact or noncompact.  The corresponding Lie subgroup \(F_0 \subset G\) is hence isogenous to a product of \(\operatorname{SU}(2)\)'s and \(\operatorname{SL}_2(\R)\)'s. As such, it contains \((F_2)^n\) as subgroup; see for example~\cite{Drutu-Kapovich:geometric}*{Section~7.7} for a proof that also \(\operatorname{SU}(2)\) contains \(F_2\).  This shows \(\operatorname{sork} \Phi(\mathfrak{g}) \le \nu(G)\) for all \(\mathfrak{g}_0\) with Cartan subalgebra \(\mathfrak{h}_0 \subset \mathfrak{k}_0\).  Of course this includes all compact real simple Lie algebras.

\subsubsection{The remaining cases}

It remains to consider those noncomplex real simple Lie algebras \(\mathfrak{g}_0\) which neither split nor have a compact Cartan subalgebra. The classification~\cite{Knapp:beyond}*{Theorem~6.105, p.\,362} reveals that only three such cases are left, meaning \(\mathfrak{g}_0\) is isomorphic to either of:

\begin{enumerate}[(I)]
\item \label{item:su-star} \(\mathfrak{su}^*(2n) = \mathfrak{sl}_n(\mathbb{H})\) with \(n \ge 2\),
\item \label{item:so-pq} \(\mathfrak{so}(p,q)\) with \(p,q\) odd and \(p + q \ge 8\),
\item \label{item:e6} \(\mathfrak{e}_{6(-26)}\).
\end{enumerate}

Cases~\eqref{item:su-star} and~\eqref{item:e6} are quickly dealt with.  The complexification of \(\mathfrak{su}^*(2n)\) is \(\mathfrak{sl}_{2n}(\C)\) and the standard Cartan involution of \(\mathfrak{su}^*(2n)\) has \(\mathfrak{k}_0 = \mathfrak{sp}(n)\) which is of complex type \(C_n\).  From Section~\ref{sec:compact-cartan}, we thus conclude
\[ \nu(\operatorname{SU}^*(2n)) \ge \nu(\operatorname{Sp}(n)) = \operatorname{sork}(C_n) = n = \operatorname{sork}({A_{2n-1}}) = \operatorname{sork}(\Phi(\mathfrak{sl}_{2n}(\C))), \]
where we used our calculation of strong orthogonal ranks of irreducible root systems in Proposition~\ref{prop:sork-of-irred}.  With the same argument, one gets
\[ \nu(\operatorname{E}_{6(-26)}) \ge \nu(F_{4(-52)}) = \operatorname{sork}(F_4) = 4 = \operatorname{sork}(E_6) = \operatorname{sork}(\Phi(\mathfrak{e}_6)) \]
where \(F_{4(-52)}\) denotes the unique compact Lie group whose complexified Lie algebra is \(\mathfrak{f}_4\). 
If \(p\) and \(q\) are odd and \(p+q = 2 \text{ mod } 4\), we obtain similarly
\begin{align*} \nu(\operatorname{SO}^0(p,q)) &\ge \nu(\operatorname{SO}(p) \times \operatorname{SO}(q)) = \frac{p-1}{2} + \frac{q-1}{2} = \frac{p+q}{2} -1 = \\
&= \operatorname{sork}\left(D_{\frac{p+q}{2}}\right) = \operatorname{sork}(\Phi(\mathfrak{so}(p+q; \C))).
\end{align*}

So the only case that still needs consideration is \(\mathfrak{g}_0 = \mathfrak{so}(p,q)\) with \(p, q\) odd and \(p + q = 0 \text{ mod } 4\).  Hence the proof of part~\eqref{item:real-form} is complete.

\subsection{Proof of part~\eqref{item:sopq}: \(\mathfrak{g}_0 = \mathfrak{so}(p,q)\) with \(p, q\) odd and \(p+q \equiv 0 \text{ mod } 4\)}

In this final case, we obtain from Proposition~\ref{prop:sork-of-irred} that for the maximal compact subgroup, we have \(\nu(\operatorname{SO}(p) \times \operatorname{SO}(q)) = \frac{p+q}{2} - 1\) but this time \(\operatorname{sork}(\Phi(\mathfrak{so}(p+q; \C))) = \operatorname{sork}(D_{\frac{p+q}{2}}) = \frac{p+q}{2}\).  So we only have to exclude the possibility \(\nu(\operatorname{SO}^0(p,q)) = \frac{p+q}{2}\).  If that was true, a free \(n\)-torus \((F_2)^n \subset G = \operatorname{SO}^0(p,q)\) with \(n = \frac{p+q}{2}\) would give rise to a subalgebra \(\mathfrak{f}_0 \subset \mathfrak{g}_0 = \mathfrak{so}(p,q)\) which is the direct sum of \(n\) simple ideals, each isomorphic to either \(\mathfrak{su}(2)\) or \(\mathfrak{sl}_2(\R)\).  Indeed, the adjoint representation embeds \(G/Z(G)\) into the \(\R\)-points of the \(\R\)-group \(\operatorname{Aut}(\mathfrak{g}_0)\) and the Lie algebras of the Zariski closures of the \(F_2\)-factors of \((F_2)^n \le G/Z(G) \subset \operatorname{Aut}(\mathfrak{g}_0)(\R)\) are real semisimple, hence contain either \(\mathfrak{su}(2)\) or \(\mathfrak{sl}_2(\R)\) as subalgebra.  Thus \(\mathfrak{f} = \mathfrak{f}_0\otimes \C \cong (\mathfrak{sl}_2(\C))^n\) is a \(\sigma\)-invariant subalgebra of \(\mathfrak{g} = \mathfrak{so}(2n; \C)\) for the complex conjugation \(\sigma\) in \(\mathfrak{g}\) with respect to \(\mathfrak{g}_0\).  Preservation of Jordan decomposition gives that a Cartan subalgebra \(\mathfrak{h}_0 \subset \mathfrak{f}_0\) is also a Cartan subalgebra in \(\mathfrak{g}_0\).  By~\cite{Knapp:beyond}*{Proposition~6.59}, the Cartan subalgebra \(\mathfrak{h}_0\) of \(\mathfrak{g}_0\) is conjugate to a \(\theta_0\)-stable one for any fixed Cartan involution \(\theta_0\) of \(\mathfrak{g}_0\).  Conjugating the Cartan involution instead of \(\mathfrak{h}_0\), we may assume that \(\mathfrak{h}_0\) is \(\theta_0\)-stable to begin with.  As we saw in Proposition~\ref{prop:full-rank-regular}, the subalgebra \(\mathfrak{f} \subset \mathfrak{g}\) is regular with respect to the \(\theta\)-stable Cartan subalgebra \(\mathfrak{h} = \mathfrak{h}_0 \otimes \C\).

As above, the Cartan involution \(\theta_0\) gives rise to the compact form \(\mathfrak{u}_0 = \mathfrak{k}_0 \oplus \textup{i} \mathfrak{p}_0\) whose complex conjugation \(\tau = \theta \circ \sigma\) acts on \(\Phi(\mathfrak{g},\mathfrak{h})\) by point reflection because~\cite{Knapp:beyond}*{Theorem~6.6, p.\,295} applies to any complex Cartan subalgebra.  Let \(\Sigma \subset \Phi(\mathfrak{g},\mathfrak{h})\) be the closed subroot system of type \((A_1)^n\) corresponding to \(\mathfrak{f}\).  Since \(\mathfrak{f}_0 = \mathfrak{f}^\sigma\) and \(\mathfrak{f} \subset \mathfrak{g}\) is regular, the involution \(\sigma\) must map each root of \(\Sigma\) either to itself or to its opposite.  For if \(\sigma\) swapped two \(A_1\)-factors in \(\Sigma\), then the regular subalgebra \(\mathfrak{r}\) determined by these two \(A_1\)-factors would satisfy \(\mathfrak{r}^\sigma \cong \mathfrak{sl}_2(\C)\) considered as real Lie algebra.  The Lie subgroup corresponding to \(\mathfrak{r}^\sigma\) would then have free subgroup rank one by what we already proved.  Thus the Lie subgroup \(F_0 \subset G\) corresponding to \(\mathfrak{f}_0\) could not have free subgroup rank~\(n\).

This shows that also \(\theta = \sigma \circ \tau\) either reflects or fixes each root in \(\Sigma\).  Roots reflected by \(\theta\) are also known as \emph{real roots}.  They must vanish on \((\mathfrak{h}_0 \cap \mathfrak{k}_0)\) so they take only real values on \(\mathfrak{h}_0\) because \(\operatorname{ad}(h)\) is self-adjoint for \(h \in (\mathfrak{h}_0 \cap \mathfrak{p}_0)\).  Roots fixed by \(\theta\) are known as \emph{imaginary roots}.  They must vanish on \((\mathfrak{h}_0 \cap \mathfrak{p}_0)\) so they take purely imaginary values on \(\mathfrak{h}_0\) because \(\operatorname{ad}(h)\) is skew-adjoint for \(h \in (\mathfrak{h}_0 \cap \mathfrak{k}_0)\).  We pick one root from each reflected \(A_1\)-factor in \(\Sigma\) and obtain a sequence of strongly orthogonal real roots.  The corresponding composition~\(D\) (in any order) of \emph{Cayley transforms} \cite{Knapp:beyond}*{Section~VI~7}
\[ d_\alpha  = \operatorname{Ad}(\exp \textup{i} \textstyle \frac{\pi}{4}(\theta (x_\alpha) - x_\alpha)) \]
with \(x_\alpha \in \mathfrak{g}_\alpha \cap \mathfrak{g}_0\) normalized by \(B_\theta(x_\alpha, x_\alpha) = \frac{2}{\alpha(h_\alpha)}\) yields a \(\theta_0\)-stable Cartan subalgebra \(\mathfrak{h}_0' = D(\mathfrak{h}) \cap \mathfrak{g}_0\) with respect to which all the roots in
\[ \Sigma^D = \{ \alpha \circ D^{-1} \,\colon\, \alpha \in \Sigma \} \subset \Phi(\mathfrak{g}, D(\mathfrak{h})) \]
are imaginary.  But since \(\Sigma^D \subset \Phi(\mathfrak{g}, D(\mathfrak{h}))\) is a subroot system of full rank, this means that \(\theta_0\) fixes \(\mathfrak{h}_0'\), or equivalently \(\mathfrak{h}_0' \subset \mathfrak{k}_0\), which is absurd because \(\mathfrak{g}_0\) has no Cartan subalgebra contained in \(\mathfrak{k}_0\).  This finishes the proof of part~\eqref{item:sopq} and hence the proof of Theorem~\ref{thm:nu-of-simple} is complete.

\begin{remark}
For the reader unfamiliar with Cayley transforms, here is an alternative conclusion of the proof. Since the subroot system \(\Sigma \subset \Phi(\mathfrak{g},\mathfrak{h})\) has full rank, the involution \(\theta\) acts as a composition of root kernel reflections on \(\Phi(\mathfrak{g},\mathfrak{h})\) and hence as an element of the Weyl group of \(\Phi(\mathfrak{g},\mathfrak{h})\).  But that implies that \(\theta\) acts identically on the Dynkin diagram of \(\Phi(\mathfrak{g},\mathfrak{h})\) because the automorphism group of a root system is the semidirect product of the Weyl group and the diagram symmetries.  Since point reflection is also an element of the Weyl group in type \(D_n\) with even \(n\), this shows that \(\sigma = \theta \circ \tau\) acts as an element of the Weyl group, contradicting that \(\mathfrak{g}_0\) is an outer form of \(\mathfrak{g}\).
\end{remark}

\end{document}